\documentclass{amsart}

\title[CR proof for an estimate for the Diederich--Fornaess index]
{A CR proof for a global estimate of the Diederich--Fornaess index of Levi-flat real hypersurfaces}

\author{Masanori Adachi}
\address{Center~for~Geometry~and~its~Applications, Pohang~University~of~Science~and~Technology, Pohang 790-784, Republic of Korea
\& 
Graduate~School~of~Mathematics, Nagoya~University, Nagoya 464-8602, Japan}
\email{adachi@postech.ac.kr, m08002z@math.nagoya-u.ac.jp}
\thanks{The author is partially supported by 
an NRF grant 2011-0030044 (SRC-GAIA) of the Ministry of Education, the Republic of Korea, and 
a JSPS Grant-in-Aid for Young Scientists (B) 26800057.
}

\date{\today}

\usepackage{amsmath,amsthm,amssymb,amsrefs,latexsym}

\newtheorem*{Thm*}{Theorem}
\newtheorem{Lem}{Lemma}

\theoremstyle{definition}
\newtheorem*{Def*}{Definition}

\theoremstyle{remark}
\newtheorem{Rem}[Lem]{Remark}
\newtheorem*{Que*}{Question}

\def\R{\mathbb{R}}
\def\C{\mathbb{C}}
\def\Cont{\mathcal{C}}
\def\rank{\mathrm{rank}}
\def\Ric{\mathrm{Ric}}

\newcommand{\base}[1]{\frac{\pa}{\pa #1}}

\def\pa{\partial}
\def\ol{\overline}
\def\opa{\overline{\partial}}
\def\bd{\partial}
\renewcommand\Re{\operatorname{Re}}

\begin{document}

\maketitle

\begin{abstract}
Yet another proof is given for a global estimate of the Diederich--Fornaess index of relatively compact domains with Levi-flat boundary, namely, the index must be smaller than or equal to the reciprocal of the dimension of the ambient space. 
This proof reveals that this kind of estimate makes sense and holds also for abstract compact Levi-flat CR manifolds.
\end{abstract}

\section{Introduction}
\label{sect:intro}
The \emph{Diederich--Fornaess index} $\eta(\Omega)$ of a $\Cont^\infty$-smoothly bounded domain $\Omega$ in a complex manifold $X$ is a numerical index on the strength of certain pseudoconvexity of its boundary $\bd\Omega$. In this paper, we consider the index in the sense that $\eta(\Omega)$ is defined to be the supremum of the exponents $\eta \in (0,1]$ admitting a $\Cont^\infty$-smooth defining function of $\bd\Omega$, say $\rho: (\bd\Omega \subset) U \to \R$, so that $-|\rho|^\eta$ is strictly plurisubharmonic in $U \cap \Omega$; if no such $\eta$ is allowed, we let $\eta(\Omega) = 0$. 

For instance, if a defining function attains $\eta = 1$, it gives a strictly plurisubharmonic defining function of $\bd\Omega$ and the boundary is strictly pseudoconvex.
The pseudoconvexity of $\bd\Omega$ is clearly necessary for $\eta(\Omega)$ to be positive; a much stronger condition is actually necessary and sufficient, the existence of a defining function $\rho$ such that the complex hessian of $-\log |\rho|$ is bounded from below by a hermitian metric of $X$ near the boundary $\bd\Omega$ as observed by Ohsawa and Sibony in \cite{ohsawa-sibony1998}.

The Diederich--Fornaess index $\eta(\Omega)$ being positive means that the boundary $\bd\Omega$ is well approximated by a family of strictly pseudoconvex real hypersurfaces from inside. The original motivation of the study of Diederich and Fornaess in \cite{diederich-fornaess1977} was to construct such an approximation on pseudoconvex domains in Stein manifolds, and the index is considered to measure certain strength of the approximation. 
Since then, the meaning of the index has been intensively studied in relation to the global regularity in the $\opa$-Neumann problem, in particular, pathologies occurring on the worm domain. See for example  \cite[\S1]{fu-shaw2014}, \cite{berndtsson-charpentier} and references therein.

\medskip

Under such circumstances, Fu and Shaw \cite{fu-shaw2014} and Brinkschulte and the author \cite{adachi-brinkschulte2014} reached a general estimate for the Diederich--Fornaess index of weakly pseudoconvex domains. 
Here we state the result in a restricted form, for domains with $\Cont^\infty$-smooth Levi-flat boundary: 

\begin{Thm*}[\cite{fu-shaw2014}, see also \cite{adachi-brinkschulte2014} and \cite{adachi_aspm}]
Let $\Omega$ be a relatively compact domain with $\Cont^\infty$-smooth Levi-flat boundary $M$ in 
a complex manifold of dimension $(n+1) \geq 2$. 
Then the Diederich--Fornaess index $\eta(\Omega)$ must be $\leq 1/(n+1)$.  
\end{Thm*}

The purpose of this paper is to give yet another proof of Theorem via an estimate on the Levi-flat boundary $M$ without looking inside $\Omega$ directly. 
The idea is to identify the usual Diederich--Fornaess index $\eta(\Omega)$ with its counterpart $\eta(M)$ on  the Levi-flat boundary 
based on the author's previous work \cite{adachi_localformula}.

\begin{Def*}
Let $M$ be an oriented $\Cont^\infty$-smooth Levi-flat CR manifold. 
The \emph{Diederich--Fornaess index} $\eta(M)$ of $M$ is defined to be the supremum of $\eta \in (0, 1]$
admitting a $\Cont^\infty$-smooth hermitian metric $h^2$ of the holomorphic normal bundle $N^{1,0}_M$ of $M$ so that 
\[
i\Theta_{h} - \frac{\eta}{1-\eta} i \alpha_h \wedge \ol{\alpha_h} > 0 
\]
holds on $M$ as quadratic forms on the holomorphic tangent bundle $T^{1,0}_M$ of $M$; if no such $\eta$ is allowed, we let $\eta(M) = 0$. 
Here the forms $\alpha_h$ and $\Theta_h$ denote the leafwise Chern connection form and its curvature  form of $N^{1,0}_M$ with respect to $h^2$ respectively. (See \S\ref{sect:prelim} for their precise definitions.) 
\end{Def*}

In our setting,  $\eta(\Omega)$ agrees with $\eta(M)$ as we will see in Lemma \ref{dfindex}, 
and Theorem follows from the following main lemma.

\begin{Lem}
\label{main}
Let $M$ be a compact $\Cont^\infty$-smooth Levi-flat CR manifold of dimension $(2n+1) \geq 3$.  
Then the Diederich--Fornaess index $\eta(M)$ must be $\leq 1/(n+1)$.
\end{Lem}

\medskip

The organization of this paper is as follows. 
In \S\ref{sect:prelim}, we provide preliminaries on CR geometry and 
confirm that the two notions of Diederich--Fornaess index, $\eta(\Omega)$ and $\eta(M)$, actually coincide for Levi-flat real hypersurfaces based on previous works.   
In \S\ref{sect:proof}, after proving Lemma \ref{main}, we give a remark that the substantial content of Lemma \ref{main} has been already pointed out by Bejancu and Deshmukh \cite{bejancu_deshmukh-1996} in manner of differential geometry, and conclude this paper with an open question.

\subsection*{Acknowledgements}
The author gratefully acknowledges an enlightening discussion with J. Brinkschulte. He is also grateful to T. Inaba for his useful remarks.

\section{Preliminaries}
\label{sect:prelim}

\subsection{Levi-flat CR manifold}
Let us recall the notion of Levi-flat CR manifold briefly. 
In the sequel, ``smooth'' means infinitely differentiable.

Let $M$ be a smooth manifold of dimension $(2n+1) \geq 3$. 
A \emph{CR structure} (of hypersurface type) of $M$ is given by a subbundle $T^{0,1}_M \subset \C \otimes_\R TM$
satisfying the following conditions:
\begin{itemize}
\item $T^{0,1}_M$ is a smooth $\C$-subbundle $T^{0,1}_M \subset \C \otimes_\R TM$ of $\rank_\C$ $n$;
\item $T^{1,0}_M \cap T^{0,1}_M = $ (the zero section) where $T^{1,0}_M := \{ v \in  \C \otimes_\R TM \mid \ol{v} \in T^{0,1}_M\}$;
\item $[ \Gamma(T^{0,1}_M), \Gamma(T^{0,1}_M)] \subset \Gamma(T^{0,1}_M)$
\end{itemize}
where $\Gamma(\,\cdot\,)$ denotes the space of smooth sections of the bundle, and the bracket means the Lie bracket of complexified vector fields. 
The pair $(M, T^{0,1}_M)$ is called a \emph{CR manifold}, which is regarded as an abstraction of real hypersurfaces in complex manifolds 
associated with their (anti-)holomorphic tangent bundles.

We say that a CR manifold $(M, T^{0,1}_M)$ is \emph{Levi-flat} if it satisfies further integrability condition
\begin{equation}
\label{leviflat}
[ \Gamma(T^{1,0}_M \oplus T^{0,1}_M), \Gamma(T^{1,0}_M \oplus T^{0,1}_M)] \subset \Gamma(T^{1,0}_M \oplus T^{0,1}_M).
\end{equation}
This is equivalent to say that the real distribution $H_M := \Re (T^{1,0}_M \oplus T^{0,1}_M) \subset TM$ of $\rank_\R$ $2n$ is integrable in the sense of Frobenius. 
It follows from Frobenius' theorem and Newlander--Nirenberg's theorem that 
the distribution $H_M$ defines a smooth foliation $\mathcal{F}$ by complex manifolds on $M$, 
namely, we have an atlas consisting of foliated charts. We call $\mathcal{F}$ the \emph{Levi foliation}. 

For a Levi-flat CR manifold $(M, T^{0,1}_M)$, we shall refer to $T^{1,0}_M$ as the \emph{holomorphic tangent bundle} of $M$ and 
call the quotient $\C$-line bundle $N^{1,0}_M$, 
\[
0 \to T^{1,0}_M \oplus T^{0,1}_M \to \C \otimes_\R TM \overset{\pi}{\to} N^{1,0}_M \to 0, 
\]
the \emph{holomorphic normal bundle}.  
This is because $T^{1,0}_M$ agrees with the holomorphic tangent bundle of the leaves of the Levi foliation $\mathcal{F}$.
Note that our holomorphic tangent bundle is distinct from 
$(\C \otimes_\R TM)/T^{0,1}_M$ and 
our \emph{$(p,q)$-form} on $M$ means a section of 
$\bigwedge^p (T^{1,0}_M)^* \otimes \bigwedge^q (T^{0,1}_M)^* \subset \bigwedge^{p+q} (T^{1,0}_M \oplus T^{0,1}_M)^*$.

Now let us consider a Levi-flat CR manifold, simply denoted by $M$, 
and define the form $\alpha_h$ mentioned in \S\ref{sect:intro}.
Fix a smooth hermitian metric $h^2$ of $N^{1,0}_M$; in our convention, we denote by
$h: N^{1,0}_M \to \R$ the map given by the norm induced from $h^2$ on $(N^{1,0}_M)_p$ for each $p \in M$.
Pick a local smooth section $\xi$ of $N^{1,0}_M$ around $p \in M$ so that it is both normalized by $h^2$ and real, i.e., $\ol{\xi} = \xi$,
which is determined up to its sign. Using such a $\xi$, we define the $(1,0)$-form $\alpha_h: T^{1,0}_M \to \C$ so as to satisfy
\begin{equation}
\label{bott}
\pi([v, \widetilde{\xi}]_p) = -\alpha_h(v_p) \xi_p
\end{equation}
for $v_p \in (T^{1,0}_M)_p$ where $\widetilde{\xi}$ and $v$ are any lift and extension of $\xi$ and $v_p$ to local sections of $\C \otimes_\R TM$ respectively.
Here we used the Levi-flatness (\ref{leviflat}) to assure that $\alpha_h$ is independent of the choice of $\xi$, $\widetilde{\xi}$ and $v$. 
We define $\ol{\alpha}_h (\ol{v_p}) := \ol{\alpha_h(v_p)}$, the complex-conjugate $(0,1)$-form of $\alpha_h$.

\begin{Rem}
The left hand side of (\ref{bott}) is the covariant derivative of $\xi$ along $v_p$ with respect to a complex Bott connection of the Levi foliation $\mathcal{F}$
and the form $\alpha$ is considered to measure the size of infinitesimal holonomy of $\mathcal{F}$ with respect to $h^2$. 
\end{Rem}

We give the $(1,1)$-form $\Theta_h: T^{1,0}_M \otimes T^{0,1}_M \to \C$ by 
\begin{align*}
\Theta_h(v_p \otimes \ol{w}_p) 
&:= v_p \alpha_h(\ol{w}) - \ol{w}_p \alpha_h(v) - \alpha_h([v, \ol{w}]_p) \\
&= - \ol{w}_p \alpha_h(v) - \alpha_h([v, \ol{w}]_p)
\end{align*}
where ${v}$ and $\ol{w}$ are arbitrary extensions of $v_p$ and $\ol{w}_p$ to local sections of $T^{1,0}_M $ and $T^{0,1}_M$ respectively. 
We again used the Levi-flatness (\ref{leviflat}) for the last term to be defined.

\subsection{Description on foliated charts}
\label{foliation}
Although we have defined the forms $\alpha_h$ and $\Theta_h$ in a coordinate-free manner, 
their descriptions on foliated charts are convenient in actual computations. Here we briefly introduce them. 

Take a \emph{foliated chart} $(U, (z_U, t_U))$ of the Levi-flat CR manifold $M$, a chart  $(z_U, t_U): U \to \C^n \times \R$ so that 
$T^{1,0}_M|U$ agrees with the pull-back bundle of $T^{1,0}\C^n \subset \C \otimes_\R T(\C^n \times \R)$. 
Any coordinate change between intersecting foliated charts, say $(U, (z_U, t_U))$ and $(V, (z_V, t_V))$, are of the form
\[
z_U = z_U(z_V, t_V), \quad t_U = t_U(t_V)
\]
where $z_U$ is holomorphic in $z_V$. A \emph{leaf} $N$ of $\mathcal{F}$ is a connected complex manifold 
injectively immersed in $M$ such that $z_U$ is holomorphic and $t_U$ is locally constant on $U \cap N$ 
for any foliated chart $(U, (z_U, t_U))$.
Our manifold $M$ is decomposed into the direct sum of the leaves of $\mathcal{F}$. 
A \emph{CR function} on $M$, a $\C$-valued function which is annihilated by vectors in $T^{0,1}_M$ by its definition, 
agrees with a function which is \emph{leafwise} holomorphic, namely, 
holomorphic in $z_U$ on any foliated chart $(U, (z_U, t_U))$. 

On a foliated chart $(U, (z_U = (z^1_U, z^2_U, \cdots, z^n_U), t_U))$, we may trivialize $T^{1,0}_M$ and $N^{1,0}_M$ by using 
\[
\left\{ \base{z^1_U}, \base{z^2_U}, \cdots, \base{z^n_U} \right\} \quad \text{and} \quad \base{t_U}
\]
respectively. This description illustrates that $T^{1,0}_M$ and $N^{1,0}_M$ are \emph{locally trivial CR vector bundles}, 
smooth vectors bundles with local trivialization covers whose transition functions are CR.
The transition functions of $N^{1,0}_M$ are much better; They are leafwise constant. 

Some computations show that on a foliated chart $(U, (z_U, t_U))$, the forms $\alpha_h$ and $\Theta_h$ 
for a given hermitian metric $h^2$ of $N^{1,0}_M$ are described as
\begin{align*}
\alpha_h &= \sum_{j=1}^n \frac{\pa \log h_U}{\pa z^j_U} dz^j_U,  \\
\Theta_h &= \sum_{j,k=1}^n \frac{\pa^2 (-\log h_U)}{\pa z^j_U \pa \ol{z}^k_U} dz^j_U \wedge d\ol{z}^k_U
\end{align*}
where $h_U := h(\base{t_U})$. 
We can see that $\alpha_h$ and $\Theta_h$ agree with the leafwise Chern connection and curvature form of $N^{1,0}_M$
with respect to $h^2$ respectively up to a positive multiplicative constant.

\subsection{The Diederich--Fornaess index}

In this subsection, we confirm that the two notions of Diederich--Fornaess index given in \S\ref{sect:intro} coincide for Levi-flat real hypersurfaces.

Let $\Omega$ be a relatively compact domain with smooth Levi-flat boundary $M$ in a complex manifold of dimension $\geq 2$. 
We introduce here terms for intermediate notions that appeared in the definition of the Diederich--Fornaess indices.
The \emph{Diederich--Fornaess exponent} $\eta_\rho$ of a fixed defining function $\rho: (\bd\Omega \subset) U \to \R$ of $\bd\Omega$ 
is the supremum of the exponents $\eta \in (0,1]$ such that $-|\rho|^\eta$ is strictly plurisubharmonic in $U \cap \Omega$; if no such $\eta$ is allowed, we let $\eta_\rho = 0$. 
We also define the \emph{Diederich--Fornaess exponent} $\eta_h$ of a fixed hermitian metric $h^2$ of $N^{1,0}_M$ in the same manner. 
The Diederich--Fornaess indices are clearly the supremum of the corresponding Diederich--Fornaess exponents.

\begin{Lem}
\label{dfindex}
We have $\eta(\Omega) = \eta(M)$.
\end{Lem}

\begin{proof}
It is proved in \cite[Theorem 1.1]{adachi_localformula} that 
one can construct a smooth hermitian metric $h^2_\rho$  of $N^{1,0}_M$ 
from a given smooth defining function $\rho$ of $M$ with $\eta_\rho > 0$
so that $\eta_\rho = \eta_{h_\rho}$.
Hence, $\eta(\Omega) \leq \eta(M)$. 

To derive the other inequality, it suffices to show that any hermitian metric $h^2$ of $N^{1,0}$ with $\eta_h > 0$,
which condition is equivalent to $i\Theta_h > 0$ as quadratic forms on $T^{1,0}_M$, can be obtained by the construction above from a defining function of $M$.
This inverse construction originates from the work of Brunella \cite{brunella2008} where he proved that 
this is possible if the Levi foliation of $M$ extends to a holomorphic foliation on a neighborhood of $M$. 
Although the extended holomorphic foliation may not exist in our setting, 
we are able to apply refined constructions explained in \cite[\S1]{ohsawa2013}, \cite[Proposition 1]{biard-iordan}, 
or \cite[Proposition 3.1]{adachi_aspm} and finish the proof.
\end{proof}

\begin{Rem}
We have restricted ourselves not to formulate the results for Levi-flat real hypersurfaces with finite differentiability 
because we have a technical problem at this point. 
The construction from defining functions to hermitian metrics in \cite{adachi_localformula} loses one order in differentiability 
since taking its normal derivative, 
although the inverse constructions in \cite{biard-iordan} or \cite{adachi_aspm} do not give us a gain in differentiability.
So we cannot simply state that any $\Cont^k$-smooth hermitian metric can be obtained from a $\Cont^{k}$ or $\Cont^{k+1}$-smooth defining function
for $2 \leq k < \infty$ unlike in the case $k = \infty$. 
\end{Rem}

\section{The Proof of Lemma \ref{main} and a Remark}
\label{sect:proof}
\subsection{Proof of Lemma \ref{main}}
Now we shall give the proof of Lemma \ref{main}.
\begin{proof}[Proof of Lemma \ref{main}]
Suppose the contrary: $\eta(M) > 1/(n+1)$. By definition, there exists a smooth hermitian metric of $N^{1,0}_M$, say $h^2$, such that
\[
i\Theta_{h} - \frac{1}{n} i \alpha_h \wedge \ol\alpha_h > 0
\]
as quadratic forms on $T^{1,0}_M$. 

By taking a double covering of $M$ if necessary, we may assume that $M$ is oriented. 
We let $\eta := h_U dt_U$ where $t_U$ is the transverse coordinate of a positively-oriented foliated chart $(U, (z_U,t_U))$ and $h_U := h(\base{t_U})$. 
Then we see that $\eta$ is a well-defined 1-form on $M$,
and that $\Theta_h \wedge \eta$, $\alpha_h \wedge \eta$ and $\ol{\alpha_h} \wedge \eta$ make sense as differential forms on $M$
regardless of the choice of extensions of $\alpha_h$ or $\Theta_h$ to tensors on $\C \otimes_\R TM$. 
Among these forms, we can show the equalities $(d\alpha_h) \wedge \eta = \Theta_h \wedge \eta$ and $d\eta = (\alpha_h + \ol{\alpha_h}) \wedge \eta$ 
from straightforward computation on the foliated chart. 

Now we obtain by direct computation that 
\begin{align*}
& d \left( (i\Theta_{h} - \frac{1}{n} i \alpha_h \wedge \ol{\alpha_h})^{n-1} \wedge i\alpha_h \wedge \eta \right) \\
&= (n-1) (i\Theta_{h} - \frac{1}{n} i \alpha_h \wedge \ol{\alpha_h})^{n-2} \wedge  \frac{1}{n} i \Theta_h \wedge i\alpha_h \wedge \ol{\alpha_h} \wedge \eta  \\
& \quad + (i\Theta_{h} - \frac{1}{n} i \alpha_h \wedge \ol{\alpha_h})^{n-1} \wedge (i\Theta_h - i\alpha_h \wedge \ol{\alpha_h})\wedge \eta \\
&= (i\Theta_{h} - \frac{1}{n} i \alpha_h \wedge \ol{\alpha_h})^{n} \wedge \eta,
\end{align*}
and Stokes' theorem yields a contradiction: 
\begin{align*}
0 &< \int_M (i \Theta_{h} - \frac{1}{n} i \alpha_h \wedge \ol{\alpha_h})^{n} \wedge \eta \\
   &= \int_M d \left( (i\Theta_{h} - \frac{1}{n} i \alpha_h \wedge \ol{\alpha_h})^{n-1} \wedge i\alpha_h \wedge \eta \right) \\
   &= 0.
\end{align*}
\end{proof}

\begin{Rem}
The proof shows in particular that $\int_M i\Theta_h \wedge \eta = \int_M i\alpha_h \wedge \ol{\alpha}_h \wedge \eta$ always holds when $\dim_\R M = 3$. 
This equality well explains the behavior of the Diederich--Fornaess exponent of an explicit example described in \cite[\S5]{adachi_aspm}.
\end{Rem}

\subsection{The approach of Bejancu and Deshmukh}
We give a remark that the substantial content of Lemma \ref{main} has been already observed by Bejancu and Deshmukh \cite{bejancu_deshmukh-1996} 
in the context of differential geometry. 
\begin{Rem}
When $\dim_\R M = 3$, the integrand $(i\Theta_{h} - i \alpha_h \wedge \ol{\alpha_h}) \wedge \eta$ was used in \cite{bejancu_deshmukh-1996} to show that the totally real Ricci curvature of compact Levi-flat real hypersurfaces in K\"ahler surfaces cannot be everywhere positive. 

Let us explain this coincidence. Suppose that we have an oriented smooth Levi-flat real hypersurface $M$ in a K\"ahler surface $(X, \omega)$. 
We restrict on $M$ the K\"ahler metric $\omega$ as a Riemannian metric and 
consider its Levi-Civita connection $\nabla^M$ and Ricci curvature $\Ric^M$. 
We also consider the Gauss--Kronecker curvature $G_{\mathcal{F}/M}$ of the leaves of the Levi foliation $\mathcal{F}$ in $M$. 
Take the signed distance function to $M$ with respect to the given K\"ahler metric $\omega$ and induce a hermitian metric $h^2$ of $N^{1,0}_M$ from it. 
Then, we can observe by direct computation that 
\begin{align*}
4(i\Theta_{h} - i \alpha_h \wedge \ol{\alpha_h})
& = (\Ric^M(\xi, \xi) - 2G_{\mathcal{F}/M})\, \omega|T^{1,0}_M \otimes T^{0,1}_M  \\
& = (\Ric^M(\xi, \xi)  - \frac{1}{2}\| d\eta \|^2 + \|\nabla^M\xi\|^2)\, \omega|T^{1,0}_M \otimes T^{0,1}_M
\end{align*}
where $\xi$ is the Reeb vector field of $M$ chosen so that it is normalized and orthogonal to $H_M$ with respect to $\omega$ and positively directed, and $\eta$ is the metric dual of $\xi$. 
The last line is exactly the integrand used in \cite{bejancu_deshmukh-1996}. 
We leave the details of this computation to the reader, who can find the techniques needed in \cite{adachi-brinkschulte_ricci} and  \cite{bejancu_deshmukh-1996}.
\end{Rem}

\subsection{Open Question}
We conclude this paper with stating an open question explicitly.

\begin{Que*}
Can we formulate the Diederich--Fornaess index for any CR manifold of hypersurface type?
Needless to say, it should agree with the Diederich--Fornaess index of its complemental domain when it is realized as the boundary real hypersurface 
of a domain in a complex manifold. 
Can we prove the global estimate of Fu and Shaw, and Brinkschulte and the author for this index in its full generality?
\end{Que*}


\begin{bibdiv}
\begin{biblist}

\bib{adachi_localformula}{article}{
   author = {Adachi, Masanori},
   title = {A local expression of the Diederich--Fornaess exponent and the exponent of conformal harmonic measures},
   status = {to appear in Bull. Braz. Math. Soc. (N.S.)},
   eprint={arXiv:1403.3179},
}

\bib{adachi_aspm}{article}{
   author = {Adachi, Masanori},
   title = {On a global estimate of the Diederich--Fornaess index of Levi-flat real hypersurfaces},
   status = {to appear in Adv. Stud. Pure. Math.},
   eprint = {arXiv:1410.2693},
}

\bib{adachi-brinkschulte2014}{article}{
   author = {Adachi, Masanori},
   author = {Brinkschulte, Judith},
   title = {A global estimate for the Diederich--Fornaess index of weakly pseudoconvex domains},
   status = {to appear in Nagoya Math. J.},
   eprint = {arXiv:1401.2264},
}

\bib{adachi-brinkschulte_ricci}{article}{
   author = {Adachi, Masanori},
   author = {Brinkschulte, Judith},
   title = {Curvature restrictions for Levi-flat real hypersurfaces in complex projective planes},
   status = {preprint},
   eprint = {arXiv:1410.2695},
}

\bib{bejancu_deshmukh-1996}{article}{
   author={Bejancu, Aurel},
   author={Deshmukh, Sharief},
   title={Real hypersurfaces of ${\bf C}{\rm P}^n$ with non-negative
   Ricci curvature},
   journal={Proc. Amer. Math. Soc.},
   volume={124},
   date={1996},
   number={1},
   pages={269--274},
}

\bib{berndtsson-charpentier}{article}{
   author={Berndtsson, Bo},
   author={Charpentier, Philippe},
   title={A Sobolev mapping property of the Bergman kernel},
   journal={Math. Z.},
   volume={235},
   date={2000},
   number={1},
   pages={1--10},
}


\bib{biard-iordan}{article}{
   author={Biard, S{\'e}verine},
   author={Iordan, Andrei},
   title={Non existence of Levi flat hypersurfaces with positive normal bundle in compact K\"ahler manifolds of dimension $\geq 3$},
   status={preprint},
   eprint={arXiv:1406.5712},
}


\bib{brunella2008}{article}{
   author={Brunella, Marco},
   title={On the dynamics of codimension one holomorphic foliations with
   ample normal bundle},
   journal={Indiana Univ. Math. J.},
   volume={57},
   date={2008},
   number={7},
   pages={3101--3113},
}

\bib{diederich-fornaess1977}{article}{
   author={Diederich, Klas},
   author={Fornaess, John Erik},
   title={Pseudoconvex domains: bounded strictly plurisubharmonic exhaustion
   functions},
   journal={Invent. Math.},
   volume={39},
   date={1977},
   number={2},
   pages={129--141},
}


\bib{fu-shaw2014}{article}{
   author={Fu, Siqi},
   author={Shaw, Mei-Chi},
   title={The Diederich-Forn{\ae}ss exponent and non-existence of Stein domains with Levi-flat boundaries},
   journal={J. Geom. Anal.},
   status={published online on 25 November 2014},
   doi={10.1007/s12220-014-9546-6}
}

\bib{ohsawa-sibony1998}{article}{
   author={Ohsawa, Takeo},
   author={Sibony, Nessim},
   title={Bounded p.s.h. functions and pseudoconvexity in K\"ahler manifold},
   journal={Nagoya Math. J.},
   volume={149},
   date={1998},
   pages={1--8},
}

\bib{ohsawa2013}{article}{
   author={Ohsawa, Takeo},
   title={Nonexistence of certain Levi flat hypersurfaces in K\"ahler
   manifolds from the viewpoint of positive normal bundles},
   journal={Publ. Res. Inst. Math. Sci.},
   volume={49},
   date={2013},
   number={2},
   pages={229--239},
}


\end{biblist}
\end{bibdiv}
\end{document}